\documentclass[10pt,a4paper]{amsart}
\usepackage{graphicx,multirow,array,amsmath,amssymb}

\newtheorem{theorem}{Theorem}

\newtheorem{proposition}[theorem]{Proposition}
\newtheorem{lemma}[theorem]{Lemma}
\newtheorem{corollary}[theorem]{Corollary}

\begin{document}

\title{Mosaic number of knots}

\author[H. J. Lee]{Hwa Jeong Lee}
\address{Department of Mathematical Sciences, KAIST, 291 Daehak-ro, Yuseong-gu, Daejeon 305-701, Korea}
\email{hjwith@kaist.ac.kr}
\author[K. Hong]{Kyungpyo Hong}
\address{Department of Mathematics, Korea University, Anam-dong, Sungbuk-ku, Seoul 136-701, Korea}
\email{cguyhbjm@korea.ac.kr}
\author[H. Lee]{Ho Lee}
\address{Department of Mathematical Sciences, KAIST, 291 Daehak-ro, Yuseong-gu, Daejeon 305-701, Korea}
\email{figure8@kaist.ac.kr}
\author[S. Oh]{Seungsang Oh}
\address{Department of Mathematics, Korea University, Anam-dong, Sungbuk-ku, Seoul 136-701, Korea}
\email{seungsang@korea.ac.kr}

\thanks{2010 Mathematics Subject Classification: 57M25, 57M27, 81P15, 81P68}
\thanks{The corresponding author(Seungsang Oh) was supported by Basic Science Research Program through
the National Research Foundation of Korea(NRF) funded by the Ministry of Science,
ICT \& Future Planning(MSIP) (No.~2011-0021795).}
\thanks{This work was supported by the National Research Foundation of Korea(NRF) grant
funded by the Korea government(MEST) (No. 2011-0027989).}

\begin{abstract}
Lomonaco and Kauffman developed knot mosaics to give a definition of a quantum knot system.
This definition is intended to represent an actual physical quantum system.
A knot $n$-mosaic is an $n \times n$ matrix of 11 kinds of specific mosaic tiles representing a knot or a link.
The mosaic number $m(K)$ of a knot $K$ is the smallest integer $n$
for which $K$ is representable as a knot $n$-mosaic.
In this paper we establish an upper bound on the mosaic number of a knot or a link $K$
in terms of the crossing number $c(K)$.
Let $K$ be a nontrivial knot or a non-split link except the Hopf link.
Then $m(K) \leq c(K) + 1$.
Moreover if $K$ is prime and non-alternating except $6^3_3$ link, then $m(K) \leq c(K) - 1$.
\end{abstract}

\maketitle

\section{Knot mosaics} \label{sec:mosaic}
Much of the theory of knots is closely related to quantum physics.
Lomonaco and Kauffman introduced a knot mosaic system to give a definition of a quantum knot system
which can be viewed as a blueprint for the construction of an actual physical quantum system
in the series of papers \cite{K, L, LK1, LK2, LK3, LK4}.
This paper was inspired from open question (8) in \cite{LK2}.

Throughout this paper we will frequently use the term ``knot" to mean either a knot or a link
for simplicity of exposition.
The following $11$ symbols are called {\em mosaic tiles\/}; \\

\begin{figure}[ht]
\includegraphics[scale=0.7]{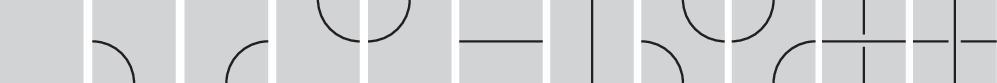}
\label{fig0}
\end{figure}

Let $n$ be a positive integer.
We define an {\em $n$-mosaic\/} as an $n \times n$ matrix $M=(M_{ij})$ of mosaic tiles.
A {\em knot $n$-mosaic\/} is an $n$-mosaic in which each curve segment on a mosaic tile is
suitably connected together on both sides with other curve segments on mosaic tiles immediately next to
in either the same row or the same column,
and whose boundary does not have end-points of curve segments.
Then a knot $n$-mosaic represents a specific knot.
One natural question concerning knot mosaics may be to determine the size of matrices representing knots.
Define the {\em mosaic number\/} $m(K)$ of a knot $K$ as the smallest integer $n$
for which $K$ is representable as a knot $n$-mosaic.
Four examples of mosaics in Figure \ref{fig1} are a 4-mosaic, the Hopf link 4-mosaic,
the trefoil knot 4-mosaic and $6^3_3$ 6-mosaic.

\begin{figure}[h]
\begin{center}
\includegraphics[scale=0.57]{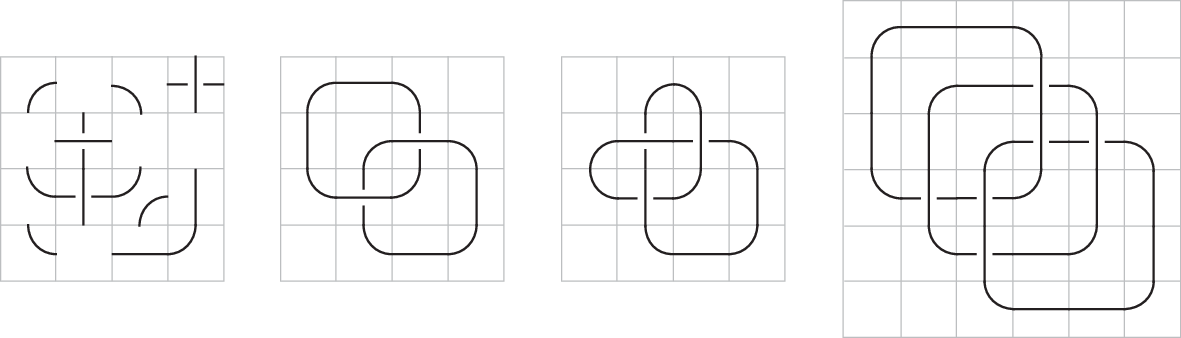}
\end{center}
\caption{Four examples of mosaics}
\label{fig1}
\end{figure}

As an analog to the planar isotopy moves and the Reidemeister moves for standard knot diagrams,
Lomonaco and Kauffman created for knot mosaics the $11$ mosaic planar isotopy moves
and the mosaic Reidemeister moves.
They conjectured that for any two tame knots $K_1$ and $K_2$,
and their arbitrary chosen mosaic representatives $M_1$ and $M_2$, respectively,
$K_1$ and $K_2$ are of the same knot type if and only if
$M_1$ and $M_2$ are of the same knot mosaic type.
This means that tame knot theory and knot mosaic theory are equivalent.
Lomonaco-Kauffman conjecture was proved being true by Kuriya and Shehab \cite{KS}.

Lomonaco and Kauffman also proposed several open questions related to knot mosaics.
One natural question is the following;
{\em Is this mosaic number related to the crossing number of a knot?\/}
In this paper we establish an upper bound on the mosaic number of a knot $K$
in terms of its crossing number $c(K)$.
Note that the mosaic number of the Hopf link is 4,
and the prime and non-alternating $6^3_3$ link is 6,
even though their crossing numbers are 2 and 6, respectively.

\begin{theorem} \label{thm:main}
Let $K$ be a nontrivial knot or a non-split link except the Hopf link.
Then $m(K) \leq c(K) + 1$.
Moreover if $K$ is prime and non-alternating except $6^3_3$ link, then $m(K) \leq c(K) - 1$.
\end{theorem}

Another proposed question is related to the enumeration of knot mosaics.
Let $D_n$ denote the total number of knot $n$-mosaics.
Already known is that $D_1=1$, $D_2=2$ and $D_3=22$.
Recently the authors established lower and upper bounds \cite{HLLO1};
$$\frac{2}{275}(9 \cdot 6^{n-2} + 1)^2 \cdot 2^{(n-3)^2} \leq D_n
\leq \frac{2}{275}(9 \cdot 6^{n-2} + 1)^2 \cdot (4.4)^{(n-3)^2}.$$
The authors also presented the exact number of $D_n$ for small $n=4,5,6$ \cite{HLLO2}, 
the state matrix recursion algorithm producing the exact enumeration of general $D_n$
that uses recursion formula of state matrices \cite{OHLL},
and more precise bounds of the quadratic exponential growth ratio of $D_n$ \cite{Oh}.

\section{Arc index and grid diagrams}

There is an open-book decomposition of $\mathbb{R}^3$
which has open half-planes as pages and the standard $z$-axis as the binding axis.
We may regard each page as a half-plane $H_{\theta}$ at angle $\theta$
when the $x$-$y$ plane has a polar coordinate.
It can be easily shown that every knot $K$ can be embedded in an open-book decomposition
with finitely many pages so that it meets each page in a simple arc.
Such an embedding is called an {\em arc presentation\/} of $K$.
The {\em arc index\/} $\alpha(K)$ is defined to be the minimal number of pages
among all possible arc presentations of $K$.

We introduce two theorems which are crucial in the proof of the main theorem.
Bae and Park established an upper bound on arc index in terms of crossing number.
Corollary $4$ and Theorem $9$ in \cite{BP} provide the following;

\begin{theorem} {\textup{(Bae-Park)}} \label{thm:BP}
Let $K$ be a knot or a non-split link.
Then $\alpha(K) \leq c(K)+2$.
Moreover if $K$ is prime and non-alternating, then $\alpha(K) \leq c(K)+1$.
\end{theorem}

Later Jin and Park improved the second part of the above theorem as Theorem $3.3$ in \cite{JP}.

\begin{theorem} {\textup{(Jin-Park)}} \label{thm:JP}
Let $K$ be a non-alternating prime knot or link.
Then $\alpha(K) \leq c(K)$.
\end{theorem}

A {\em grid diagram\/} is a link diagram of vertical strands and the same number of horizontal strands
with the properties that at every crossing the vertical strand crosses over the horizontal strand
and no two horizontal segments are co-linear and no two vertical segments are co-linear.
It is known that every knot admits a grid diagram \cite{C}.
The minimal number of vertical segments in all grid diagrams of a knot $K$ is called
the {\em grid index\/} of $K$, denoted by $g(K)$.
Since grid diagrams are a way for depicting arc presentations \cite{C},
we will think of the grid index and the arc index equivalently, i.e. $\alpha(K)=g(K)$.
Three figures in Figure \ref{fig2} show an arc presentation of the trefoil knot,
a grid diagram, and how they are related.
Note that both of the arc index and the grid index of the trefoil knot are 5.

\begin{figure} [h]
\begin{center}
\includegraphics[scale=0.8]{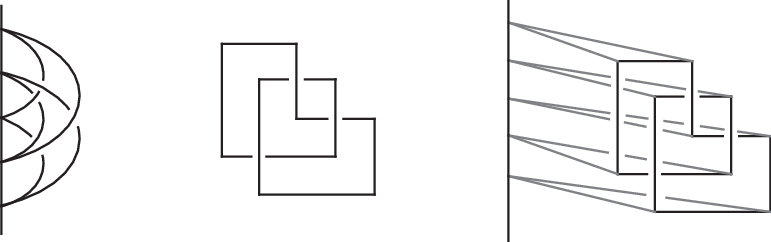}
\end{center}
\caption{An arc presentation and a grid diagram of the trefoil knot}
\label{fig2}
\end{figure}

Dynnikov introduced the following properties of grid diagram system in \cite{D}.
In this paper we use cyclic permutations only.

\begin{proposition} {\textup{(Dynnikov)}} \label{prop:D}
Two grid diagrams of the same link can be obtained from each other by a finite sequence of
the following elementary moves.
\begin{itemize}
\item cyclic permutation of horizontal (vertical) edges;
\item stabilization and destabilization;
\item interchanging neighbouring edges if their pairs of endpoints do not interleave.
\end{itemize}
\end{proposition}

\begin{figure}[h]
\begin{center}
\includegraphics[scale=0.55]{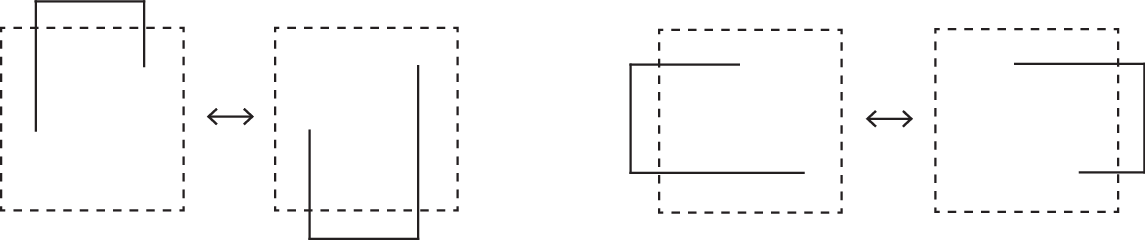}
\end{center}
\caption{Cyclic permutations on a grid diagram}
\label{fig3}
\end{figure}

\section{Upper bound on the mosaic number}

In this section we will prove Theorem \ref{thm:main}.
Let $K$ be a nontrivial knot or a non-split link except the Hopf link and $6^3_3$ link.
We start with a grid diagram with the grid index $g(K)$ (which is equal to $\alpha(K)$).
We can regard this grid diagram as a knot mosaic representative of $K$ by smoothing each corner
as in Figure \ref{fig4}.
This fact guarantees that $m(K) \leq \alpha(K)$.
We distinguishes into two cases;
\vspace{2mm}

\begin{figure}[h]
\begin{center}
\includegraphics[scale=0.8]{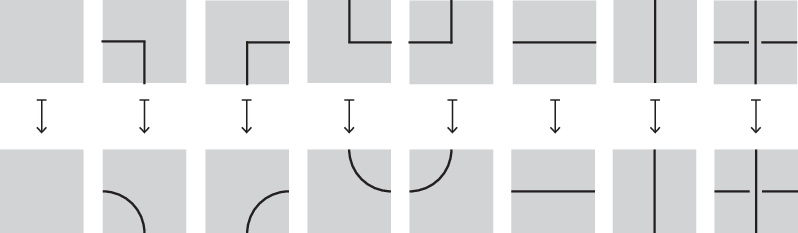}
\end{center}
\caption{From a grid diagram to a knot mosaic representative}
\label{fig4}
\end{figure}

\noindent {\bf Case 1.} The grid diagram contains a non-rectangular component.

We will find an upper bound on the mosaic number in terms of the arc index with the size reduced by 1.
By repeating cyclic permutations of horizontal edges properly,
we may assume that the horizontal edge, say $h_t$, on the top of this new grid diagram
is a part of the non-rectangular component.
Note that cyclic permutations do not change the grid index.
Let $v_l$ and $v_s$ be two vertical edges adjacent with $h_t$ where $v_l$ is longer than $v_s$.
Indeed they do not have the same length, otherwise they must be parts of a rectangular component.
Let $h_s$ be the horizontal edge adjacent with the shorter vertical edge $v_s$ other than $h_t$,
and $v_m$ the vertical edge adjacent with $h_s$ other than $v_s$.
See Figure \ref{fig5}.
By repeating cyclic permutations of vertical edges properly,
we can find a grid diagram of $K$ such that $v_s$ lies between $v_l$ and $v_m$.

\begin{figure}[h]
\begin{center}
\includegraphics[scale=0.66]{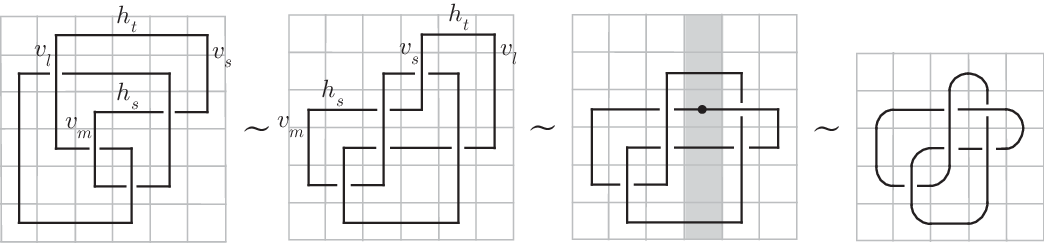}
\end{center}
\caption{Reduction of the size by 1}
\label{fig5}
\end{figure}

Now we slide down $h_t$ until it reaches to $h_s$ on the grid diagram,
keeping that $h_t$ crosses over all vertical edges.
Since this sliding uses a combination of planar isotopy moves and Reidemeister moves,
it does not change the knot type.
Even though it is not a grid diagram anymore, it is still a mosaic representative of $K$.
Finally we can delete one column of mosaic tiles which contains the vertical edge $v_s$
as the shaded region in the figure.

The result is an $(\alpha(K)-1)$-mosaic representative of $K$.
Thus we get $m(K) \leq c(K) + 1$, and moreover $m(K) \leq c(K) - 1$ for a non-alternating prime $K$
by Theorem \ref{thm:BP} and Theorem \ref{thm:JP}.
\vspace{2mm}

\noindent {\bf Case 2.} All components of the grid diagram are rectangles.

This means that $K$ is a link consisting of $n$ trivial knots for some positive integer $n$.
It is easy to see the inequality $m(K) \leq 2n$.
Jin, Kim and Ko proved the following lemma in Lemma 1 and 7 in \cite{JKK} which is very useful.

\begin{lemma} {\textup{(Jin-Kim-Ko)}} \label{lem:link}
Let $K_D$ be a link diagram with $p$ crossings, $n$ components and $r$ split components.
Then $p \geq 2(n-r)$.
Furthermore if $K_D$ is a diagram with $2n-2$ crossings of a non-split link,
then $K_D$ is a connected sum of $n-1$ Hopf links.
\end{lemma}

Let $K_D$ be a diagram with $c(K)$ crossings of $K$.
By Lemma \ref{lem:link}, $c(K) \geq 2n-2$,
and if the equality holds, then $K$ is a connected sum of $n-1$ Hopf links.
When $K$ has 3 components, the left mosaic representative in Figure \ref{fig6}
shows $m(K) \leq c(K)+1$.
If $K$ has more components, then for each component
the crossing number is increased by 2 and we can easily find a mosaic representative
whose mosaic number is increased by 2 or less.

\begin{figure}[h]
\begin{center}
\includegraphics[scale=0.6]{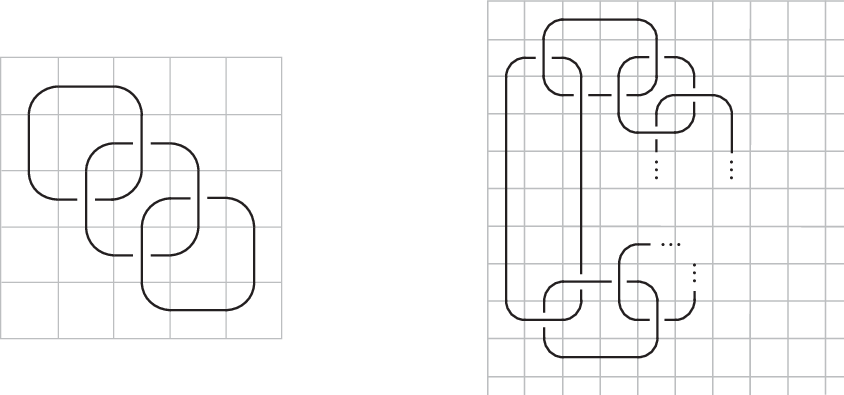}
\end{center}
\caption{All components are rectangles}
\label{fig6}
\end{figure}

If $c(K) > 2n$, then $m(K) \leq 2n \leq c(K)-1$, so we have done.
In the case of $c(K)=2n-1$ or $2n$, clearly $m(K) \leq c(K)+1$.
We only need to check the case where $K$ is prime and non-alternating.
If $c(K)=2n-1$, some component of $K_D$ must have a self-crossing.
We can nullify the crossing without increasing the number of components.
Then this new diagram is again a connected sum of $n-1$ Hopf links.
This guarantees that the original link $K$ is not prime.
Finally we may assume that $c(K)=2n$.
Each component does not have self-crossings.
Otherwise, there must be at least 2 self-crossings.
Then we can repeat the argument above to show that $K$ is not prime.
If the diagram $K_D$ has a component with only 2 crossing points,
then it must be a split link or a connected sum with the Hopf link, so a non-prime link.
Therefore each component has at least 4 crossing points.
But if one of them has more than 4 crossings, then $c(K) > 2n$.
Thus each component has exactly 4 crossings.
Since $K$ is non-splitting, this diagram looks like a necklace with $n$ rings.
So we can find a mosaic representative with $m(K) \leq 2n-2$ for $n \geq 4$
as the right in Figure \ref{fig6}.
Note that if $n=2$, then a grid diagram with 2 rectangles represents the unlink or the Hopf link.
And if $n=3$, only $6^3_3$ link is non-alternating.

\section{sharper upper bounds for several knot classes}

In this section we will present sharper upper bounds on the mosaic number of
pretzel knots and torus knots.

\begin{corollary} \label{cor:pretzel}
Let $P(-p,q,r)$ be a pretzel knot of type $(-p,q,r)$ with $p,q,r \geq 2$.
Then
\begin{itemize}
\item[(1)] $m(P(-2,q,r)) \leq q+r$ for $q,r \geq 3$;
\item[(2)] $m(P(-p,2,r)) \leq p+r+1$ for $p,r \geq 3$;
\item[(3)] $m(P(-p,3,r)) \leq p+r+1$ for $p,r \geq 3$;
\item[(4)] $m(P(-p,q,r)) \leq p+q+r-3$ for $p \geq 3$ and $q,r \geq 4$.
\end{itemize}
\end{corollary}

\begin{proof}
This corollary follows directly from the result of \cite{LJ} combined with the proof of the main theorem.
\end{proof}

\begin{theorem} \label{thm:torus}
Let $T_{p,q}$ be a $(p,q)$-torus knot.
Then $m(T_{p,q}) \leq p+q-1$.
Moreover if $|p-q| \neq 1$, then $m(T_{p,q}) \leq p+q-2$.
\end{theorem}

\begin{proof}
Figure \ref{fig7} shows a grid diagram of $T_{p,q}$ with $p<q$ with grid index $p+q$.
Now we can follow the proof of the main theorem to reduce the size by 1.
Moreover if $p+2 \leq q$,
then $(p+1)^{th}$ vertical edge and $(p+1)^{th}$ horizontal edge do not share their endpoints.
Therefore we can simultaneously apply two sliding moves on the top horizontal edge
and on the rightmost vertical edge as in Figure \ref{fig8}.
\end{proof}

\begin{figure}[h]
\begin{center}
\includegraphics[scale=0.65]{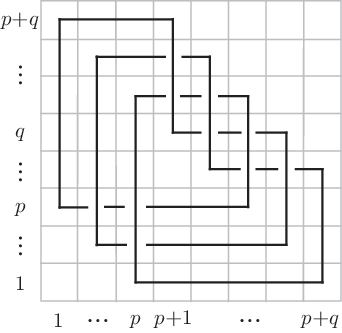}
\end{center}
\caption{Grid diagram of $T_{p,q}$ with grid index $p+q$}
\label{fig7}
\end{figure}

\begin{figure}[h]
\begin{center}
\includegraphics[scale=0.55]{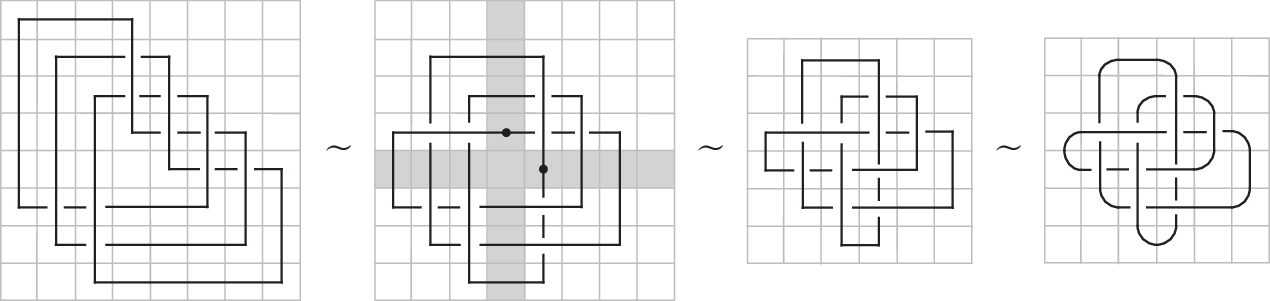}
\end{center}
\caption{Converting a grid diagram of $T_{3,5}$ into a knot 6-mosaic}
\label{fig8}
\end{figure}

\end{document}